\documentclass[a4paper,fleqn]{cas-sc}

\usepackage[numbers]{natbib}

\usepackage{hyperref}
\usepackage{cleveref}

\usepackage{algorithm}
\usepackage{algpseudocode}
\usepackage{amsthm}

\usepackage{graphicx}
\usepackage{subcaption}
\graphicspath{{figs/}}
\usepackage{placeins}
\newtheorem{theorem}{Theorem}[section]
\newtheorem{lemma}[theorem]{Lemma}

\newdefinition{definition}[theorem]{Definition}
\newdefinition{remark}[theorem]{Remark}

\def\tsc#1{\csdef{#1}{\textsc{\lowercase{#1}}\xspace}}
\tsc{WGM}
\tsc{QE}

\def\E{\mathbb{E}}
\def\reals{\mathbb{R}}

\def\range{{\rm range}}

\newcommand{\spcg}{SPCG}

\newcommand{\seq}[2]{#1_{\text{\tiny$\triangle t$}}^{#2}}

\begin{document}
\let\WriteBookmarks\relax
\def\floatpagepagefraction{1}
\def\textpagefraction{.001}

\ExplSyntaxOn
\RenewDocumentCommand \printorcid {} {}
\cs_set:Npn \__first_footerline: { \mbox{} }
\cs_new:Npn \my_initials:n #1
{
	\regex_match:nnTF { \. } { #1 }
	{ #1 }
	{
		\seq_set_split:Nnn \l_tmpa_seq { - } { #1 }
		\seq_map_indexed_inline:Nn \l_tmpa_seq
		{
			\int_compare:nNnTF { ##1 } = { 1 }
			{ \tl_head:n { ##2 }. }
			{ -\tl_head:n { ##2 }. }
		}
	}
}

\RenewDocumentCommand \eadauthor {}
{
	\seq_map_inline:Nn \l_stm_au_seq
	{
		\my_initials:n { ##1 }
	}
	\space \l_stm_au_sn_seq
}
\ExplSyntaxOff

\shorttitle{Sequential preconditioned conjugate gradient method}

\shortauthors{}

\title [mode = title]{Sequential Preconditioned Conjugate Gradient Method for Linear Statistical Models}

\author[1]{Guan-Yu Chen}

\author[1]{Dong-Yue Xie}

\author[1]{Xi Yang}
\cormark[1]
\ead{yangxi@nuaa.edu.cn}

\author[1]{Zun-Hao Zheng}

\affiliation[1]{
	organization={School of Mathematics, Nanjing University of Aeronautics and Astronautics},
	city={Nanjing},
	postcode={211106}, 
	country={China}
}
\cortext[1]{Corresponding author}

\begin{abstract}
We propose a randomized iterative method for the ordinary least-squares estimation problem in large-scale linear statistical models, namely the Sequential Preconditioned Conjugate Gradient Method (SPCG). SPCG constructs a sequence of sketched least-squares subproblems with increasing sketch sizes, applies PCG as the inner solver, and warm-starts each subproblem from the previous solution. A final refinement stage is then performed on the full-scale problem. Since most iterations are carried out on smaller subproblems, the overall computational cost is significantly reduced. We establish the convergence theory, prove that SPCG attains OLS prediction accuracy, and derive per-subproblem iteration bounds and complexity estimates. Numerical experiments show that SPCG reaches the target prediction accuracy with fewer iterations and less CPU time than full-data PCG and Iterative Double Sketching (IDS).
\end{abstract}



\maketitle

\section{Introduction}\label{sec:intro}
Consider the linear statistical model with the true parameter $\beta$ as follows
\begin{equation}\label{eq:model}
	Y=X\beta+\zeta,
\end{equation}
where $X\in\mathbb{R}^{N\times d}$ is a feature matrix of full column rank with sample size $N$ and feature size $d$, $N\gg d$, $Y\in\mathbb{R}^{N}$, and $\zeta$ is a zero-mean noise vector with covariance matrix $\sigma^{2}I_N$. The ordinary least-squares (OLS) estimator is
\begin{equation}\label{eq:ols}
	\hat\beta
	=\arg\min_{u\in\mathbb{R}^{d}}\frac{1}{2}\|Y-Xu\|_2^2
	=(X^{\top}X)^{-1}X^{\top}Y.
\end{equation}
For large-scale problems, direct factorization methods typically require $O(Nd^2)$ arithmetic operations, which can be computationally prohibitive. Classical iterative solvers avoid explicit factorization, but the iterations still involve repeated products with the full data matrix $X$, and their convergence may deteriorate severely when $X$ is ill-conditioned. These limitations have motivated randomized algorithms based on sketching techniques \cite{woodruff2014sketching,drineas2016randnla,martinsson2020randomized}, which seek to approximate the OLS estimator more efficiently by compressing the full least-squares problem into a lower-dimensional sketched problem or improve the conditioning of the original problem through efficient randomized preconditioning. Such ideas have led to substantial progress in large-scale regression, giving rise to a wide range of methods based on Sketch-and-Solve (SAS), Iterative Sketching (IS), and Sketch-and-Precondition (SAP) frameworks \cite{pilanci2016iterative,avron2010blendenpik,meng2014lsrn,lacotte2020optimal,wang2022iterative}.

The present work is closely related to the SAS and SAP paradigms. SAS replaces $(X,Y)$ with the sketched data $(SX,SY)$ and solves the resulting lower-dimensional least-squares problem, thereby substantially reducing the computational cost. As shown in \cite{raskutti2016statistical,dobriban2019asymptotics}, the asymptotic accuracy of an SAS estimator depends explicitly on the sketch size $m$. A small $m$ results in low accuracy, whereas increasing $m$ improves the estimator accuracy but also increases the cost of solving the sketched problem. Thus, SAS involves an inherent tradeoff between statistical accuracy and computational efficiency. By contrast, SAP can attain the accuracy of the OLS estimator by using a randomized sketch to construct a preconditioner and then applying an iterative solver, typically a Krylov method, to the original problem. Although the improved conditioning can accelerate convergence, each iteration still requires operations with the full data matrix $X$, which limits the overall computational savings.

These observations motivate the proposed SPCG method, which combines the low cost of SAS with the high accuracy of SAP through a two-stage sequential PCG scheme. In the first stage, SPCG efficiently constructs a carefully designed sequence of sketched least-squares subproblems with increasing sketch sizes, whose estimators provide progressively more accurate approximations to the OLS estimator. Each subproblem is solved using PCG within the SAP framework. Since the statistical accuracy of each subproblem is determined by its sketch size, only a few iterations are required to reach the corresponding accuracy level, and the resulting estimator serves as a high-quality initial point for the next, more accurate subproblem. In the second stage, the same SAP-based PCG is applied to the full problem, requiring only a few full-data iterations to attain the prediction accuracy of the OLS estimator. In this way, most iterations are performed on smaller sketched subproblems, thereby substantially reducing the overall computational cost. Using the SRHT sketching operator, we establish the convergence theory of SPCG and prove that its final prediction error reaches the OLS statistical level. Numerical experiments further show that SPCG consistently requires fewer iterations and less CPU time than Iterative Double Sketching (IDS) and full-data PCG.

The remainder of this paper is organized as follows. In \cref{sec:SPCG}, we present the SPCG method in detail. In \cref{sec:conv}, we establish the convergence properties of SPCG and analyze its computational complexity. \Cref{sec:experiments} reports the numerical experiments that verify the efficiency of our method. Finally, conclusions are drawn in \cref{sec:conclusion}.

\section{The SPCG method}\label{sec:SPCG}
\subsection{Notations and preliminaries}
Throughout the paper, $\|\cdot\|$ denotes the Euclidean norm for vectors and the spectral norm for matrices, while $\kappa(G)$ denotes the spectral condition number of a matrix $G$. Let $X=U_XD_XV_X^\top$ be the compact singular value decomposition of $X$, where $U_X\in\reals^{N\times d}$ and $V_X\in\reals^{d\times d}$ have orthonormal columns and $D_X\in\reals^{d\times d}$ is diagonal. The $d\times d$ identity matrix is denoted by $I_d$. Expectations are taken with respect to the noise $\zeta$, while probabilities involving sketching are taken with respect to the corresponding sketching matrices. All logarithms are natural.
\begin{definition}[Subspace embedding]
	Let $X\in\reals^{N\times d}$, and let $U$ be an orthonormal basis for $\range(X)$. For $\epsilon>0$, a matrix $S\in\reals^{m\times N}$ is called a $(1+\epsilon)$-subspace embedding for $\range(X)$ if
	\(
	\|U^\top S^\top S U-I_d\|\leq\epsilon.
	\)
\end{definition}
Intuitively, a subspace embedding approximately preserves the geometry of $\range(X)$ in a lower-dimensional space, allowing the original problem to be approximated efficiently by solving the sketched problem defined by $(SX,SY)$ with a controlled loss of accuracy. Classical sketching operators include dense sketching operators such as Gaussian sketching, sparse sketching operators such as CountSketch, and structured sketching operators that admit fast matrix-vector multiplication, such as the subsampled randomized Hadamard transform (SRHT) \cite{woodruff2014sketching,murray2023randomized}. Existing theoretical results show that Gaussian sketching provides strong guarantees but is computationally expensive to apply, whereas CountSketch is fast to apply but typically requires a larger sketch size to achieve comparable embedding guarantees.

In view of the above trade-off, we mainly consider the SRHT in this paper. An SRHT has the form $S=\sqrt{N/m}\,(HDP)_{\mathcal I,:}$ when $N$ is a power of two. Here $H$ is the normalized Walsh--Hadamard matrix, $D$ is a random Rademacher diagonal matrix, $P$ is a random permutation matrix, and $\mathcal I$ is a uniformly sampled set of $m$ row indices. The following lemma gives a probabilistic subspace embedding guarantee for SRHT.

\begin{lemma}[{\cite{wang2022iterative}}]\label{thm:srht}
	Let $S\in\reals^{m\times N}$ be an SRHT sketching matrix and let $U\in\reals^{N\times d}$ have orthonormal columns. For $\epsilon,\delta\in(0,1)$, if
	\(
	m\geq c\epsilon^{-2}\left(d+\log(N/\delta)\right)
	\log(ed/\delta),
	\)
	where $c>0$ is a constant, then
	\(
	\Pr\{\|U^\top S^\top S U-I_d\|>\epsilon\}\leq \delta.
	\)
\end{lemma}

\subsection{The framework}\label{sec:framework}
Let $S_i\in\reals^{m_i\times N}$ for $i=1,\ldots,K$, where $d<m_1<\cdots<m_K<N$. The $i$th sketched least-squares estimator is $\tilde{\beta}^{i}=\arg\min_{u\in\reals^d}\frac{1}{2}\|S_iXu-S_iY\|^2$, and its intrinsic prediction-error level is $\delta_i:=\E\|X(\tilde{\beta}^{i}-\beta)\|$. As the sketching dimension $m_i$ increases, $\delta_i$ generally decreases \cite{dobriban2019asymptotics}, yielding a sequence of progressively more accurate subproblems.

The SPCG method solves the original problem \eqref{eq:ols} through a two-stage sequential procedure. In the first stage, we construct a sequence of sketched subproblems defined above with increasing sketching dimensions and solve them sequentially, using the output $\beta^i$ of the $i$th subproblem as a warm start for the $(i+1)$th subproblem. The target accuracy is $\E\|X(\beta^i-\beta)\|\leq[1+o(1)]\delta_i$, so each subproblem is solved only to the intrinsic accuracy of its sketched estimator rather than to numerical precision. In the second stage, we use the high-quality approximation $\beta^K$ to initialize an iterative solve of the full problem \eqref{eq:ols}, so only a few full-data iterations are needed to attain the prediction accuracy of the OLS estimator. This sequential combination of inexpensive sketched iterations and a few full-data iterations substantially reduces the overall computational cost without sacrificing accuracy, thus providing a more reasonable strategy than conventional full-data iterative methods.

The $K$ subproblems can be constructed efficiently using a nested SRHT construction. We compute $(Z_X,Z_Y)=(HDPX,HDPY)$ once and choose nested random row-index sets $\mathcal I_1\subset\cdots\subset\mathcal I_K$ with $|\mathcal I_i|=m_i$ and $m_2/m_1=\cdots=m_K/m_{K-1}=2$. The sketched data are then obtained as $(S_iX,S_iY)=\sqrt{N/m_i}\bigl((Z_X)_{\mathcal I_i,:},(Z_Y)_{\mathcal I_i}\bigr)$. Thus, all $K$ sketched data pairs are obtained from a single transformation of $(X,Y)$, avoiding $K$ separate applications of the sketching operators. The dominant cost is therefore a single Hadamard transform, requiring $O(Nd\log N)$ operations.

The \spcg\ method uses PCG as the inner solver in both stages. Using an independent sketching matrix $\widehat S\in\reals^{r\times N}$, we form the fixed sketched Hessian $\widehat H=(\widehat SX)^\top\widehat SX$ once and reuse it as a common preconditioner for all PCG solves under the SAP framework, where $r=O(d\log d)$ and $m_1>r$. Let $a_i$ be the number of PCG iterations for the $i$th sketched subproblem, and $T^\dagger=\sum_{i=1}^K a_i$. Algorithm~\ref{Alg-SPCG} summarizes the resulting method with a total iteration budget $T$ and an initial estimator $\beta_0$. We use $\operatorname{PCG}(A,b,\widehat H,\beta_{\text{init}},a)$ to denote the iterate obtained after $a$ PCG iterations for the least-squares problem defined by $(A,b)$, starting from $\beta_{\text{init}}$ and preconditioned by $\widehat H$. One PCG iteration on the $i$th sketched system costs $\{(4m_i+13)d+2d^2\}\,\mathrm{FLOPs}$, whereas one iteration on the full system costs $\{(4N+13)d+2d^2\}\,\mathrm{FLOPs}$. Hence, the first-stage iterations are substantially cheaper when $m_i\ll N$.


\begin{algorithm}[t]
	\caption{The SPCG Algorithm}
	\label{Alg-SPCG}
	\begin{algorithmic}[1]
		\State {\bfseries Input:} $X\in\mathbb{R}^{N\times d}$,
		$Y\in\mathbb{R}^{N}$,
		$\{(S_iX,S_iY)\}_{i=1}^{K}$, $\{a_i\}_{i=1}^{K}$, $\widehat{H}\in\mathbb{R}^{d\times d}$, $\beta_0$, and $T$.
		\State $T^{\dagger}\leftarrow\sum\limits_{i=1}^{K}a_i$,
		$t\leftarrow0$
		
		\Statex
		\textemdash\textemdash\textemdash\textemdash\textemdash
		\enspace\textsc{Stage I}\enspace
		\textemdash\textemdash\textemdash\textemdash\textemdash
		
		\For{$i\leftarrow1$ {\bfseries to} $K$}
		\State $\beta_{t+a_i}\leftarrow
		\text{PCG}(S_iX,S_iY,\widehat{H},\beta_t,a_i)$
		\State $t\leftarrow t+a_i$
		\EndFor
		
		\Statex
		\textemdash\textemdash\textemdash\textemdash\textemdash
		\enspace\textsc{Stage II}\enspace
		\textemdash\textemdash\textemdash\textemdash\textemdash
		
		\State $\beta_T\leftarrow
		\text{PCG}(X,Y,\widehat{H},\beta_{T^{\dagger}},
		T-T^{\dagger})$
		\State {\bfseries Output:} $\beta_T$
	\end{algorithmic}
\end{algorithm}

We emphasize that SPCG admits several practical improvements. Its inner least-squares solver, sketching operator, and sequence of sketching dimensions can be adapted to problems with different structures. Consequently, \spcg\ can be improved by directly incorporating any new efficient iterative solver or sketching strategy.

\section{Convergence analysis of SPCG}\label{sec:conv}
This section establishes the convergence and complexity guarantees for \spcg. The proofs of all theorems presented below are provided in Appendix~\ref{app:proofs}. For the $i$th sketched subproblem, let $\beta_j^i$ denote its $j$th PCG iterate. Since each subproblem is warm-started from its predecessor, we set $\beta_0^1=\beta_0$ and $\beta_0^i=\beta_{a_{i-1}}^{i-1}$ for $i=2,\ldots,K$. Let $\mathcal T_i:=\sum_{\ell=1}^i a_\ell$, where $\mathcal T_0=0$, so that $T^\dagger=\mathcal T_K$. Recall that $\delta_i=\E\|X(\tilde\beta^i-\beta)\|$, and let $\delta_{\rm OLS}=\E\|X(\hat\beta-\beta)\|$. The following theorem establishes the convergence for each sketched subproblem.

\begin{theorem}\label{thm:stage}
	Let $N$ and $m_1$ be powers of two, $\epsilon\in(0,1/9)$, and $\delta\in(0,1)$. Suppose that $a_i\geq1$ for $i=1,\ldots,K$. If $r\geq c\epsilon^{-2}[d+\log(N/\delta)]\log(ed/\delta)$ and $m_1>r$, where $c>0$ is a constant, then, with probability at least $1-2\delta$, for any $i=1,\ldots,K$,
	\begin{equation}\label{eq:stage}
		\E\|X(\beta_{a_i}^{i}-\beta)\|
		\leq\theta^{a_i}\E\|X(\beta_0^{i}-\beta)\|+(1+\theta^{a_i})\delta_i,
		\quad \text{where $\theta=\frac14$}.
	\end{equation}
\end{theorem}
Theorem~\ref{thm:stage} shows that each sketched subproblem converges exponentially to its intrinsic prediction-error level $\delta_i$. Building on this stagewise result, the following main theorem establishes the global convergence of \spcg.

\begin{theorem}[Global convergence]\label{thm:global}
	Let \(s=T-T^\dagger\) and $\theta=1/4$. Under the assumptions of
	Theorem~\ref{thm:stage}, suppose that \(a_i\geq1\) for
	\(i=1,\ldots,K\). Then, for any \(\beta_0 \in \reals^{d}\), it holds
	with probability at least \(1-2\delta\) that
	\begin{equation}\label{eq:global-exact}
		\E\|X(\beta_T-\beta)\|
		\leq\theta^T\E\|X(\beta_0-\beta)\|
		+\sum_{i=1}^K(1+\theta^{a_i})\theta^{T-\mathcal T_i}\delta_i
		+(1+\theta^s)\delta_{\rm OLS}.
	\end{equation}
	Furthermore, if $\delta_i\leq C\delta_{\rm OLS}$ for $i=1,\ldots,K$, where $C>0$ is a constant independent of \(i\), then
	\begin{equation}\label{eq:global-ols}
		\E\|X(\beta_T-\beta)\|
		\leq\theta^T\E\|X(\beta_0-\beta)\|
		+\left[1+\theta^s\left(1+\frac{5}{3}C\right)\right]\delta_{\rm OLS}.
	\end{equation}
\end{theorem}
Theorem~\ref{thm:global} provides the main convergence guarantee for \spcg. As the number $s$ of full-data iterations increases, the first term of \eqref{eq:global-ols} converges to zero, while the second term converges to $\delta_{\rm OLS}$. The condition $\delta_i\leq C\delta_{\rm OLS}$ is used only for the simplified bound \eqref{eq:global-ols}. For a fixed finite sequence of sketched subproblems with $\delta_{\rm OLS}>0$, one may take $C=\max_{1\leq i\leq K}\delta_i/\delta_{\rm OLS}$. We next determine the stagewise iteration numbers $a_i$. For $\omega\in(0,1)$, we require the iterative solution at stage \(i\) to satisfy the stagewise accuracy condition $\E\|X(\beta_{a_i}^i-\tilde\beta^i)\|\leq\omega\delta_i$, which ensures $\E\|X(\beta_{a_i}^i-\beta)\|\leq(1+\omega)\delta_i$.

\begin{theorem}\label{thm:iteration-bounds}
	For $i=2,\ldots,K$, let $r_{i-1,i}:=\sqrt{(m_i-d)/(m_{i-1}-d)}$ and suppose the $(i-1)$th subproblem satisfies the stagewise accuracy condition. Then the $i$th subproblem satisfies the same condition if
	\[
	a_i\geq\frac{\ln\{[(1+\omega)r_{i-1,i}+1]/\omega\}}{\ln 4}>0.
	\]
\end{theorem}
Under the schedule $m_i=2m_{i-1}$, we have $r_{i-1,i}^2=2+d/(m_{i-1}-d)$, which decreases with $i$. Since the sufficient lower bound for $a_i$ is increasing in $r_{i-1,i}$, the unrounded bounds decrease from $a_2$ to $a_K$. Based on these iteration bounds, the following theorem characterizes the computational cost of \spcg.

\begin{theorem}[Computational complexity]\label{thm:complexity}
	Assume the conditions of Theorem~\ref{thm:iteration-bounds}, $m_{i+1}/m_i=2$, $m_K=N/2$, and $K=O(1)$. Let $a_i$ attain their respective lower bounds, and define $q_1:=\E\|X(\beta_0-\tilde\beta^1)\|/\delta_1$, $\bar r:=r_{1,2}$, and $\alpha:=\ln\{[1+(1+\omega)\bar r]/\omega\}/\ln 4$. If $1<q_1<1+(1+\omega)\bar r$, then $\ln(1/\omega)/\ln 4<a_i\leq\alpha$ for $i=1,\ldots,K$. If the iteration terminates when the iterate attains the noise level $\sigma$, the dominant costs of the three stages of \spcg\ are
	\begin{itemize}
		\item initialization: $Nd\log_2N$;
		\item first stage: $4\alpha Nd$;
		\item second stage:
		$4[\ln(\omega^K/\sigma)/\ln 4]Nd$,
		where $\omega>\sigma^{1/K}$.
	\end{itemize}
\end{theorem}

Theorem~\ref{thm:complexity} shows that the overall cost is dominated by the $O(Nd\log_2N)$ SRHT initialization. A sparse operator such as CountSketch can reduce the sketching cost to $O(Nd)$, although achieving comparable theoretical embedding guarantees generally requires a larger sketch size. Nevertheless, the CountSketch variant tested in Section~\ref{sec:experiments} exhibits convergence comparable to the SRHT version while substantially reducing initialization and total runtime, indicating practical performance considerably better than its conservative theoretical guarantees suggest.

\section{Numerical experiments}\label{sec:experiments}
We compare \spcg\ with IDS \cite{wang2022iterative} and full-data PCG \cite{lacotte2021fast}. For each test, we generate a Gaussian matrix $G\in\reals^{N\times d}$ and set the feature matrix $X=GD$, where $D=\operatorname{diag}(\tilde{\sigma}_1,\ldots,\tilde{\sigma}_d)$ is chosen to give the prescribed condition number $\kappa$. The entries of \(\beta\) are independently drawn from the standard Gaussian
distribution, while those of the noise vector \(\zeta\) are independently
drawn from \(\mathcal N(0,10^{-8})\). The response vector is then generated
as \(Y=X\beta+\zeta\).

The IDS step size is set to $\mu=(1-d/r)^2/(1+d/r)$ as recommended in \cite{wang2022iterative}. For IDS and \spcg, the sketch sizes satisfy $m_{i+1}/m_i=2$, with $m_1=8d$ and $m_K=N/2$. All methods use an SRHT-based sketched Hessian of size $r=6d$, and the corresponding preconditioner is applied through a QR factorization.

We measure accuracy by the prediction error $\Delta_t=\|X(\beta_t-\beta)\|^2$ and use the OLS error $\Delta=\|X(\hat\beta-\beta)\|^2\approx\sigma^2d$ as the target accuracy \cite{pilanci2016iterative}. The maximum iteration number is $T=100$, and reported results are averages over 10 independent runs unless stated otherwise.

\subsection{Iteration parameters}
We first determine the number $a_i$ of SPCG for each sketched subproblem. Figure~\ref{fig:parameter} uses $N=2^{20}$, $d=2^6$, and $\kappa=10^4$. The left panel shows the unrounded theoretical lower bounds for $a_2$ and $a_K$ over $\omega\in\{2^{-4},2^{-3},2^{-2},2^{-1},1\}$. The two curves nearly coincide, decrease as $\omega$ increases, and remain below three, indicating that only a few iterations per subproblem are needed under typical accuracy requirements. The right panel sets $a_i=\mathrm{NoI}$ uniformly across all sketched subproblems and reports the total runtime. The runtime is minimized at $\mathrm{NoI}=3$, while $\mathrm{NoI}=2$ and $4$ result in similar runtimes. We therefore set $a_i=3$ in the main experiments.

\begin{figure}[!htbp]
	\centering
	\begin{subfigure}{0.48\linewidth}
		\includegraphics[width=\linewidth]{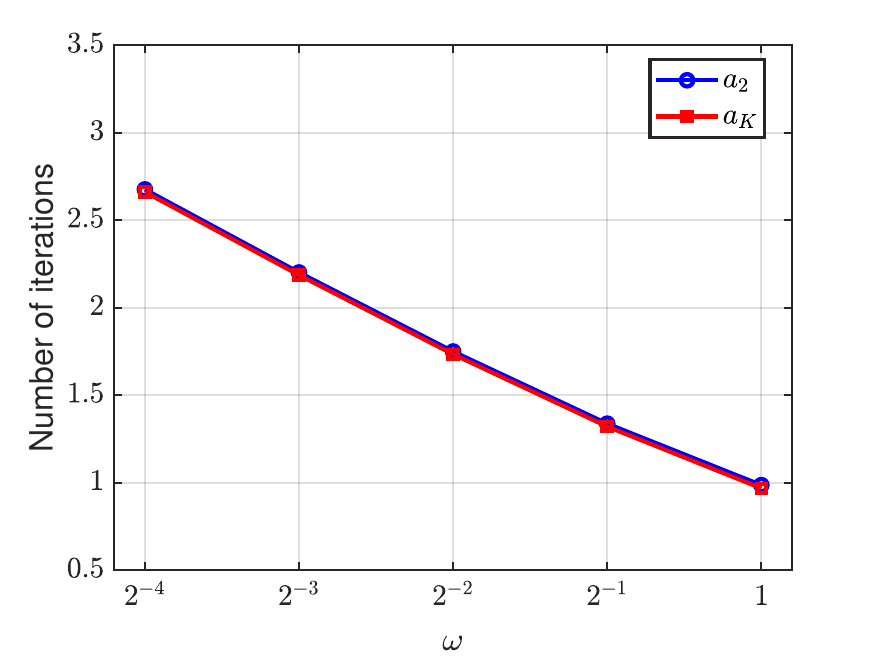}
		\caption{Lower bounds for $a_2$ and $a_K$.}
	\end{subfigure}
	\hfill
	\begin{subfigure}{0.48\linewidth}
		\includegraphics[width=\linewidth]{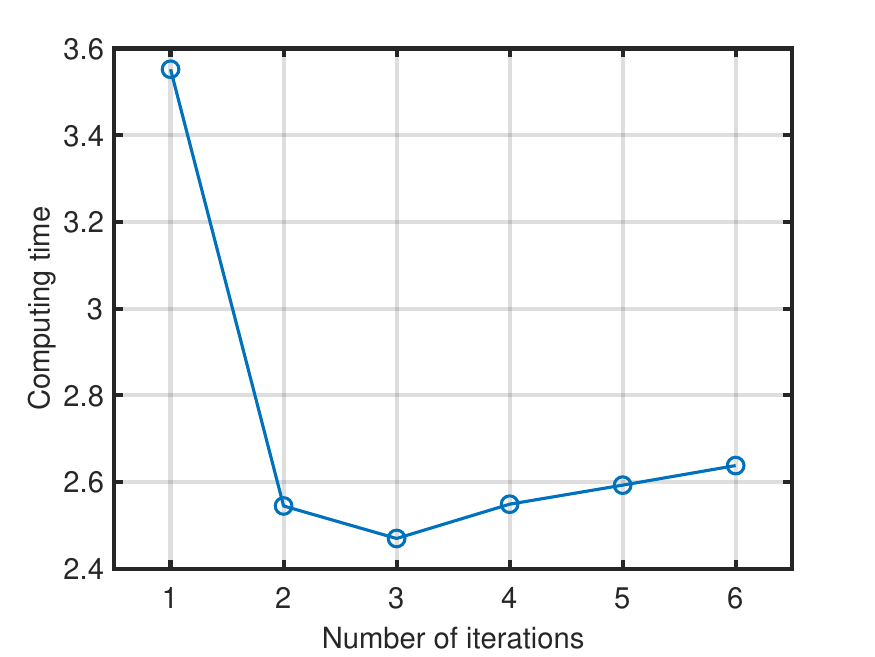}
		\caption{Runtime for $a_i=\mathrm{NoI}$.}
	\end{subfigure}
	\caption{Selection of the inner iteration numbers for \spcg.}
	\label{fig:parameter}
\end{figure}

\FloatBarrier
\subsection{Accuracy and computational performance}
We consider four problems with $N\in\{2^{17},2^{18},2^{19},2^{20}\}$, $d=2^6$, and $\kappa=10^4$, followed by two larger and more ill-conditioned problems with $N\in\{2^{21},2^{22}\}$, $d=2^6$, and $\kappa=10^8$. For $\sigma^2=10^{-8}$, the OLS target for $d=2^6$ is approximately $\sigma^2d=6.4\times10^{-7}$. We also use a two-dimensional problem with $N=2^{20}$ and $\kappa=10^4$ to visualize the randomized convergence paths.

Figure~\ref{fig:accuracy-path}(a) compares the final errors of OLS and \spcg\ for $N=2^{17},\ldots,2^{20}$; the error bars summarize the variation across repeated trials. In every case, \spcg\ reaches the OLS error level. Figure~\ref{fig:accuracy-path}(b) shows 100 independent convergence paths for \spcg\ and IDS on the two-dimensional problem. The \spcg\ paths are more concentrated, indicating that the sequential sketched estimators provide stable intermediate points.

\begin{figure}[!htbp]
	\centering
	\begin{subfigure}{0.48\linewidth}
		\includegraphics[width=\linewidth]{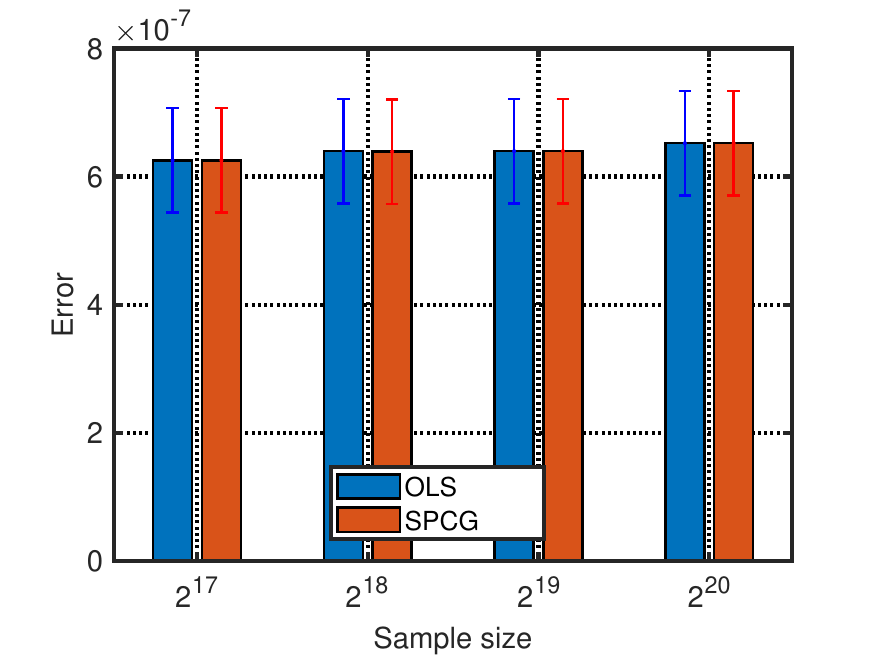}
		\caption{Prediction errors of OLS and \spcg.}
	\end{subfigure}
	\hfill
	\begin{subfigure}{0.48\linewidth}
		\includegraphics[width=\linewidth]{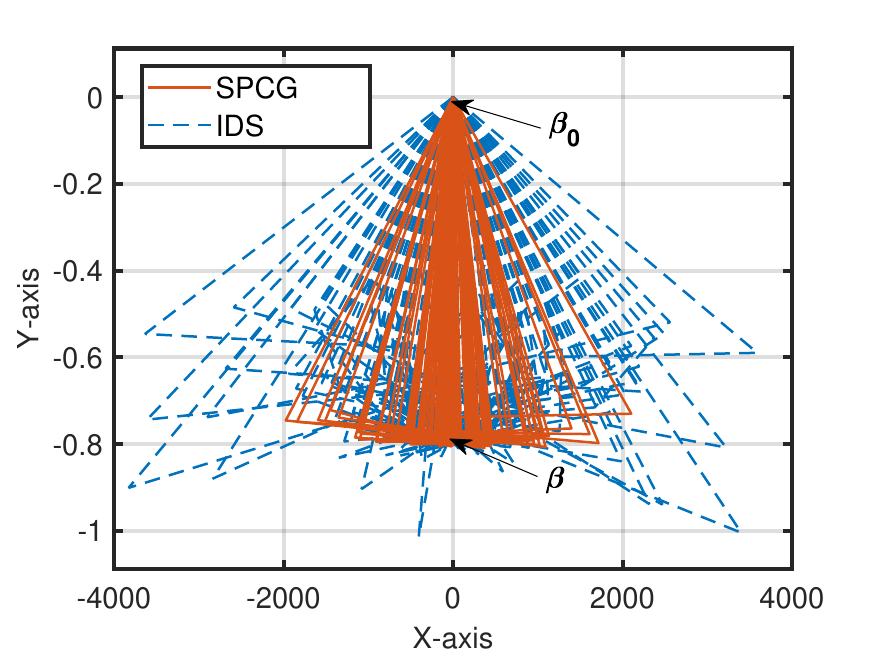}
		\caption{Two-dimensional convergence paths.}
	\end{subfigure}
	\caption{Accuracy and convergence stability of \spcg.}
	\label{fig:accuracy-path}
\end{figure}

Figure~\ref{fig:all_results} collects the prediction-error histories against iteration count and CPU time for all six test settings. For $\kappa=10^4$, \spcg\ reaches the target in substantially fewer iterations; IDS and PCG have similar convergence rates and typically require about $1.5$ times as many iterations. The runtime advantage is because most \spcg\ iterations use smaller sketched subproblems. The same behavior persists for $\kappa=10^8$. As $N$ increases to $2^{21}$ and $2^{22}$, the absolute runtime gap grows because IDS and PCG repeatedly operate on the full data. The comparison with full-data PCG also serves as an ablation of the sequential structure: both methods use PCG, but moving the early iterations to smaller sketched systems substantially reduces the total cost.

Table~\ref{tab:runtime} gives the iteration counts and CPU times required to reach the OLS target. \spcg\ is fastest in every setting. The speedup ratios in Table~\ref{tab:speedup} are defined as the baseline value divided by the corresponding \spcg\ value. Relative to IDS, the iteration and CPU speedups range from $1.38$ to $1.72$ and from $1.41$ to $1.95$, respectively. Relative to PCG, the corresponding ranges are $1.25$--$1.76$ and $2.10$--$2.83$.

\begin{figure}[htbp]
	\centering
	\begin{subfigure}{0.48\textwidth}
		\centering
		\includegraphics[width=0.48\linewidth]{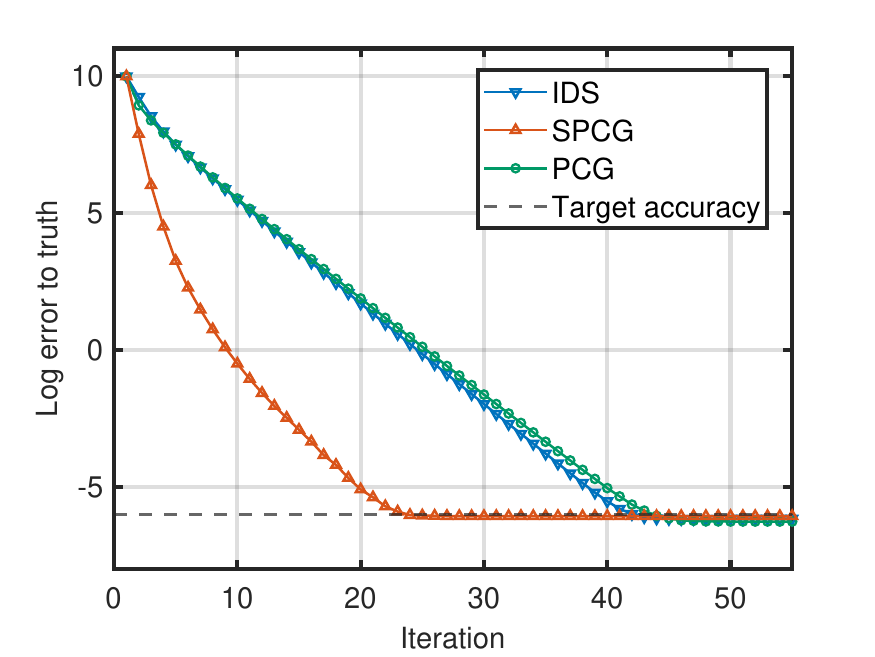}
		\includegraphics[width=0.48\linewidth]{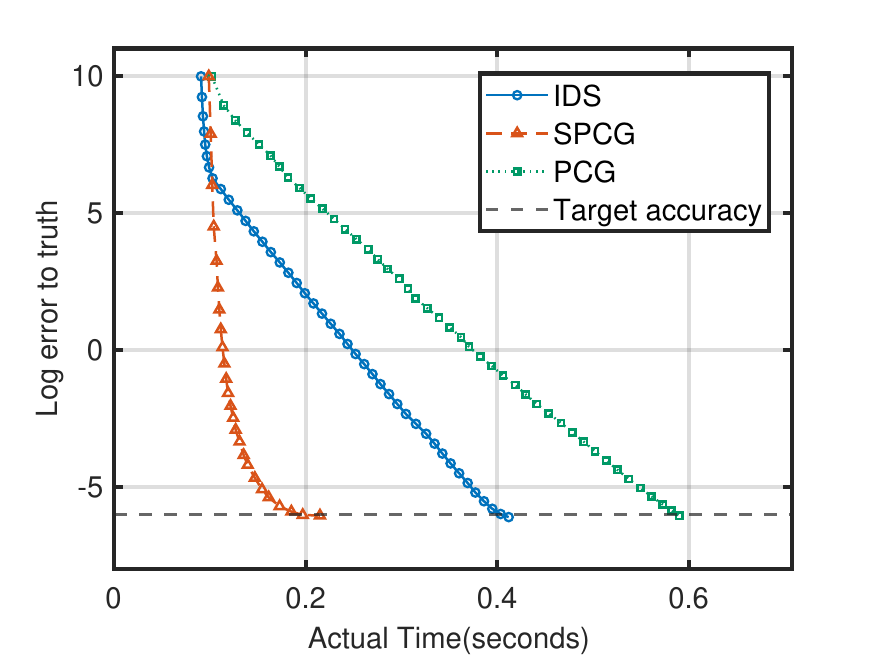}
		\caption{$N=2^{17}, d=2^6, \kappa=10^4$}
		\label{fig:group1}
	\end{subfigure}
	\begin{subfigure}{0.48\textwidth}
		\centering
		\includegraphics[width=0.48\linewidth]{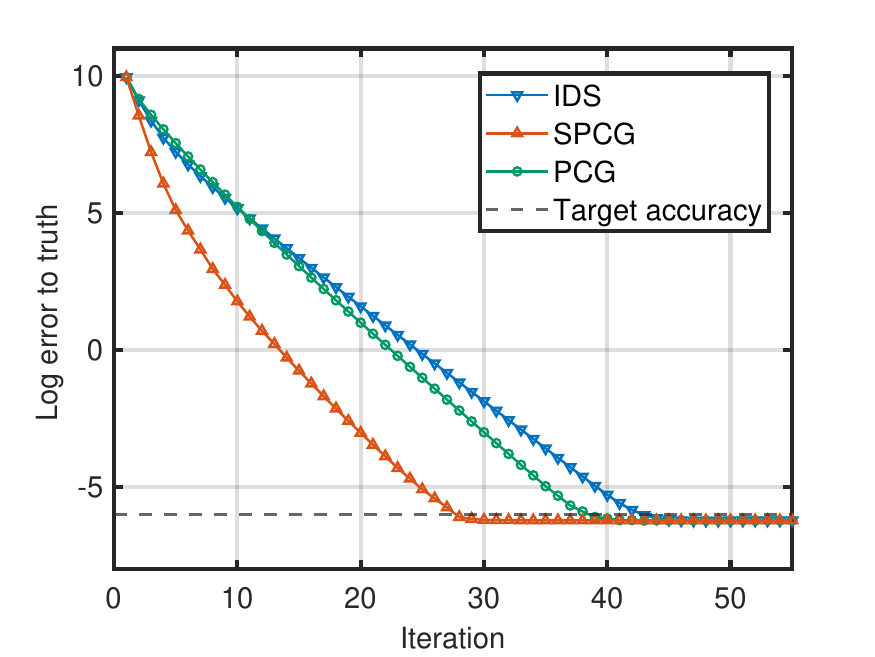}
		\includegraphics[width=0.48\linewidth]{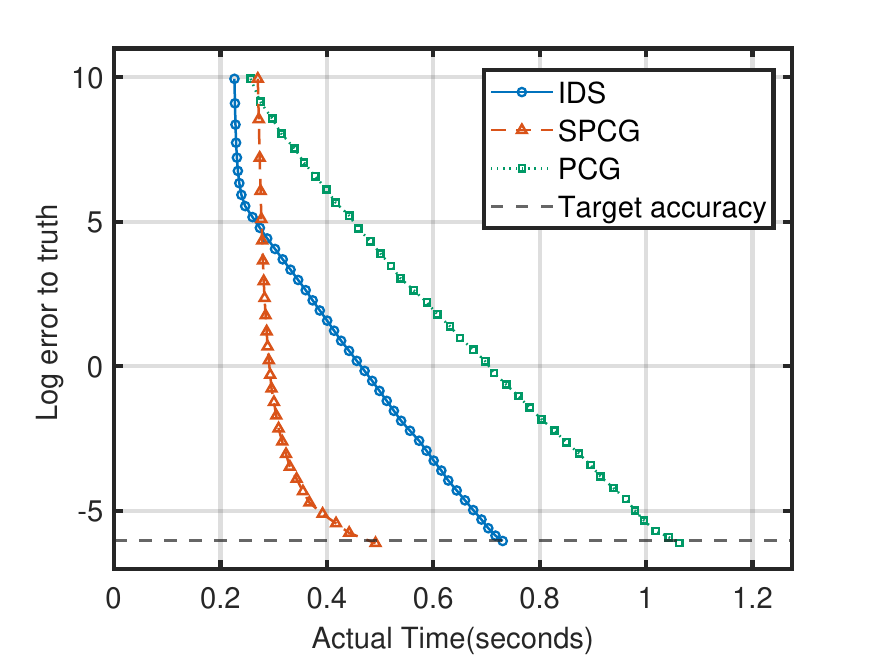}
		\caption{$N=2^{18}, d=2^6, \kappa=10^4$}
		\label{fig:group2}
	\end{subfigure}

	\begin{subfigure}{0.48\textwidth}
		\centering
		\includegraphics[width=0.48\linewidth]{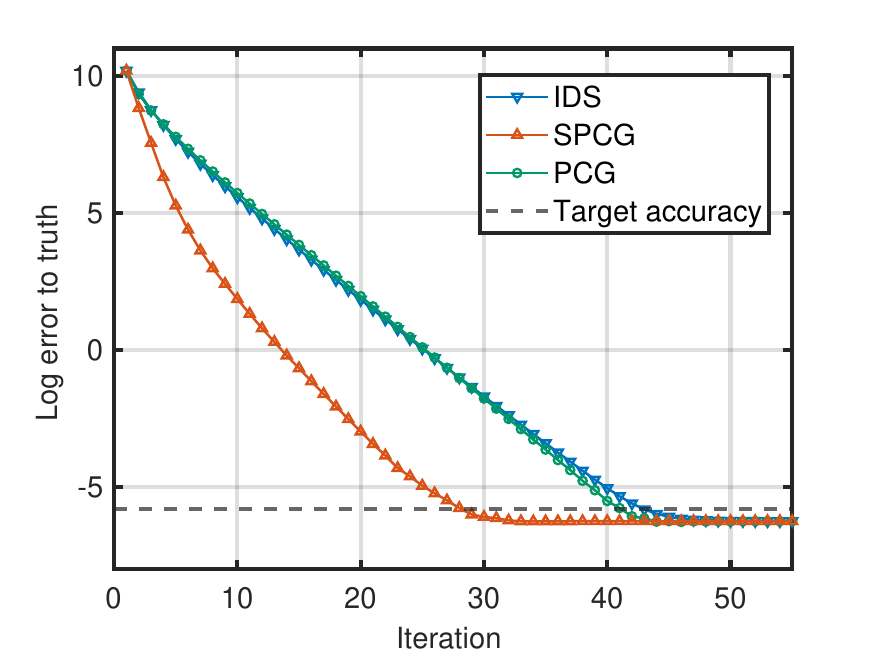}
		\includegraphics[width=0.48\linewidth]{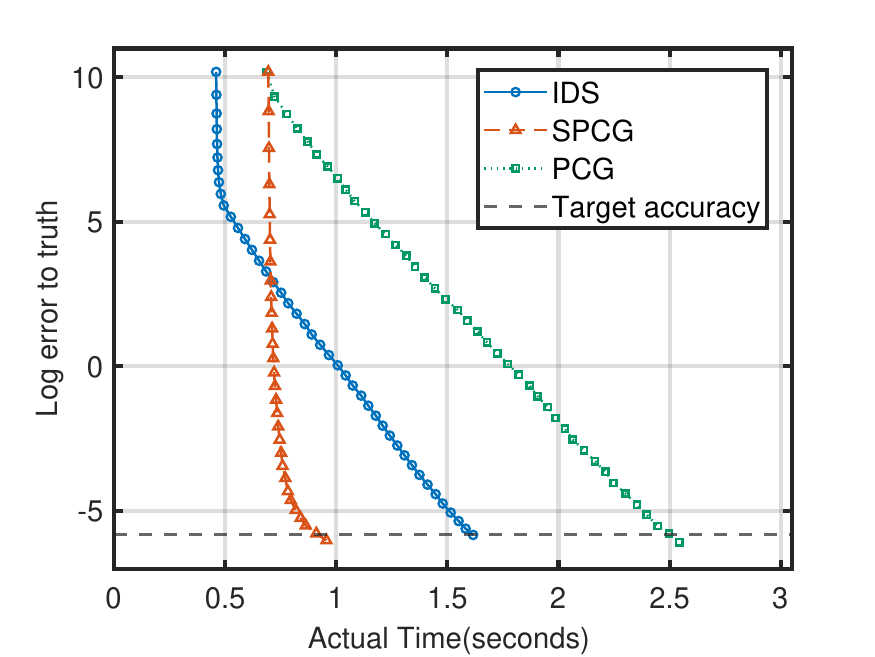}
		\caption{$N=2^{19}, d=2^6, \kappa=10^4$}
		\label{fig:group3}
	\end{subfigure}
	\begin{subfigure}{0.48\textwidth}
		\centering
		\includegraphics[width=0.48\linewidth]{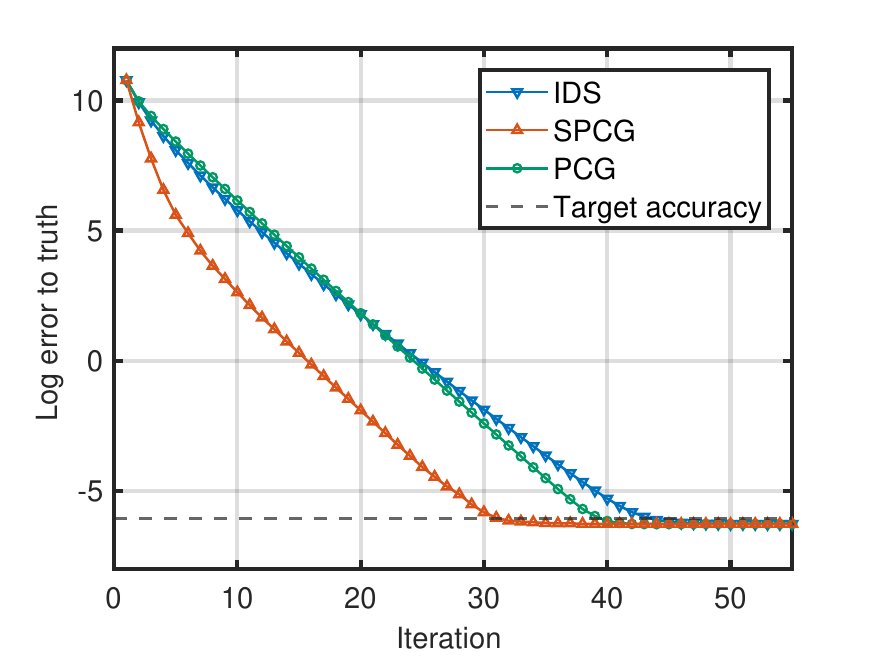}
		\includegraphics[width=0.48\linewidth]{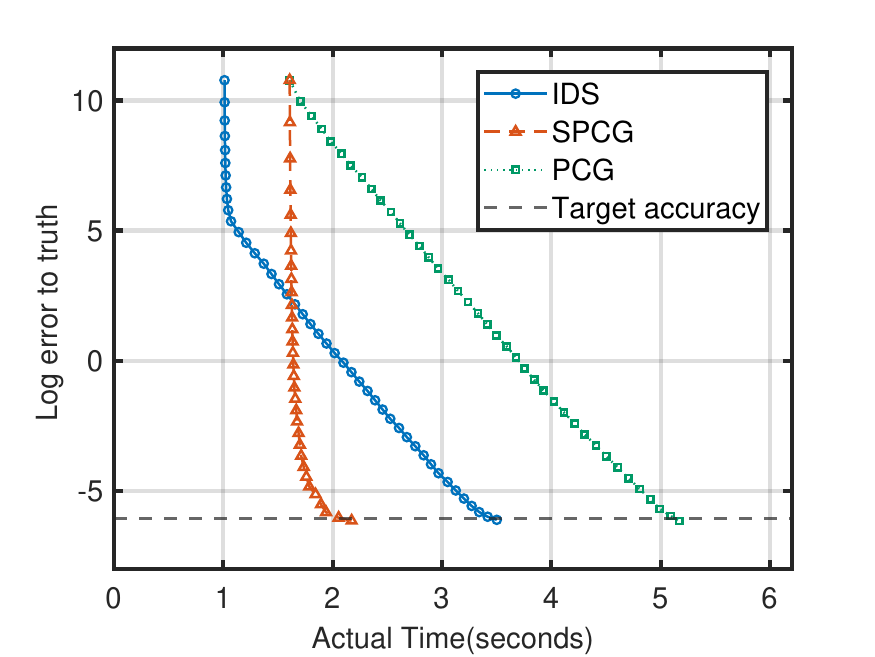}
		\caption{$N=2^{20}, d=2^6, \kappa=10^4$}
		\label{fig:group4}
	\end{subfigure}
	
	\begin{subfigure}{0.48\textwidth}
		\centering
		\includegraphics[width=0.48\linewidth]{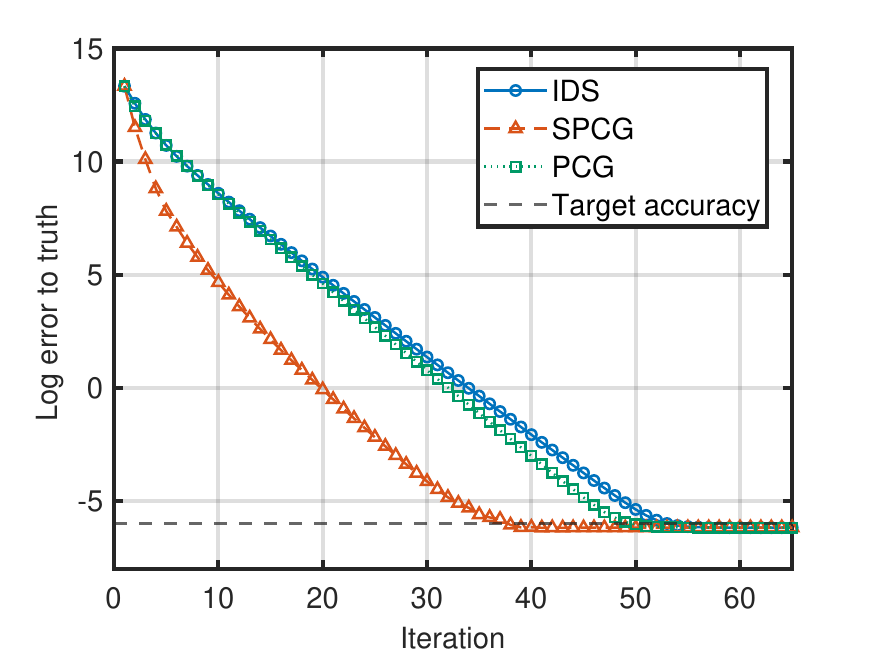}
		\includegraphics[width=0.48\linewidth]{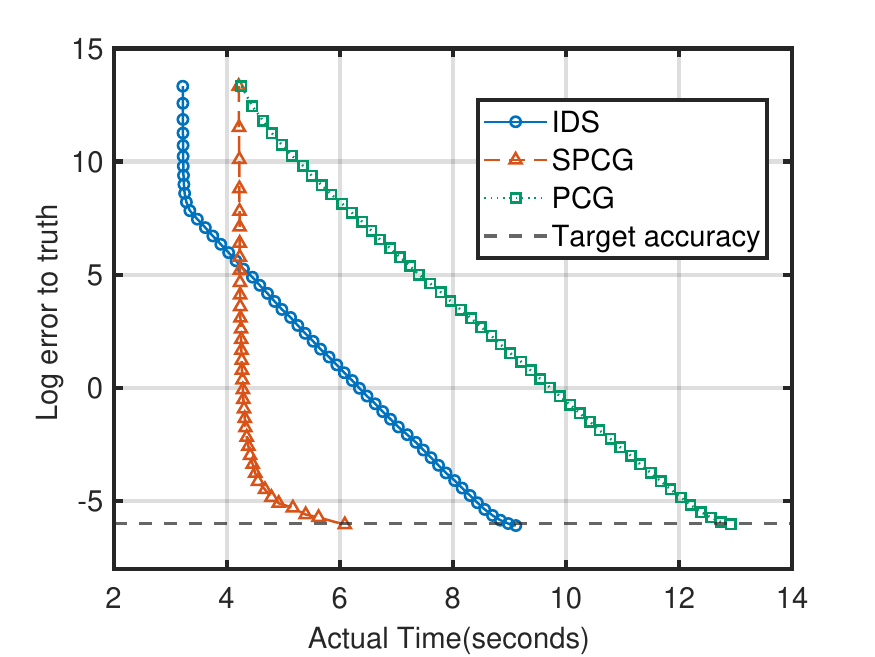}
		\caption{$N=2^{21}, d=2^6, \kappa=10^8$}
		\label{fig:group5}
	\end{subfigure}
	\begin{subfigure}{0.48\textwidth}
		\centering
		\includegraphics[width=0.48\linewidth]{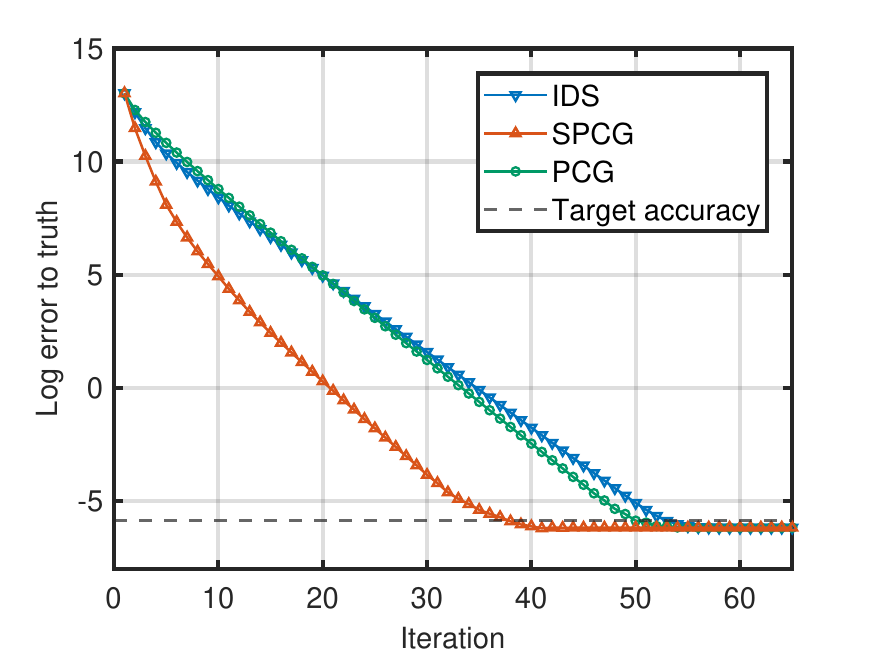}
		\includegraphics[width=0.48\linewidth]{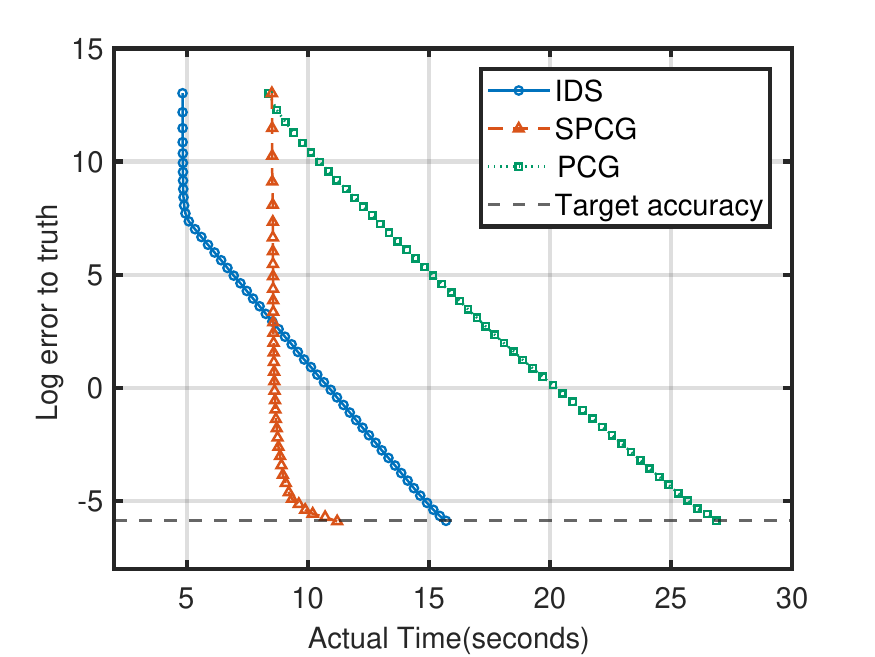}
		\caption{$N=2^{22}, d=2^6, \kappa=10^8$}
		\label{fig:group6}
	\end{subfigure}
	\caption{Prediction error and running time comparisons of different methods.}
	\label{fig:all_results}
\end{figure}

\begin{center}
	\begin{minipage}{\textwidth}
		\captionsetup{type=table}
		\caption{Iterations (IT) and CPU time for reaching the target accuracy, $d=2^6$.}
		\label{tab:runtime}
		\centering
		\resizebox{\textwidth}{!}{
			\begin{tabular}{l cc cc cc cc cc cc}
				\toprule
				& \multicolumn{8}{c}{$\kappa=10^4$} & \multicolumn{4}{c}{$\kappa=10^8$}\\
				\cmidrule(lr){2-9}\cmidrule(lr){10-13}
				Method
				& \multicolumn{2}{c}{$N=2^{17}$}
				& \multicolumn{2}{c}{$N=2^{18}$}
				& \multicolumn{2}{c}{$N=2^{19}$}
				& \multicolumn{2}{c}{$N=2^{20}$}
				& \multicolumn{2}{c}{$N=2^{21}$}
				& \multicolumn{2}{c}{$N=2^{22}$}\\
				\cmidrule(lr){2-3}\cmidrule(lr){4-5}\cmidrule(lr){6-7}
				\cmidrule(lr){8-9}\cmidrule(lr){10-11}\cmidrule(lr){12-13}
				& IT & CPU & IT & CPU & IT & CPU & IT & CPU & IT & CPU & IT & CPU\\
				\midrule
				IDS      & 43 & 4.12E-01 & 43 & 7.31E-01 & 43 & 1.61E+00 & 44 & 3.51E+00 & 54 & 9.15E+00 & 55 & 1.58E+01\\
				PCG      & 44 & 5.96E-01 & 39 & 1.06E+00 & 42 & 2.54E+00 & 40 & 5.16E+00 & 51 & 1.27E+01 & 51 & 2.68E+01\\
				\spcg & 25 & 2.11E-01 & 28 & 4.91E-01 & 29 & 9.55E-01 & 32 & 2.17E+00 & 38 & 6.04E+00 & 39 & 1.12E+01\\
				\bottomrule
		\end{tabular}}
	\end{minipage}
\end{center}

\begin{center}
	\begin{minipage}{\textwidth}
		\captionsetup{type=table}
		\caption{Speedup ratios of \spcg\ over IDS and PCG, $d=2^6$.}
		\label{tab:speedup}
		\centering
		\begin{tabular}{c c c c c c}
			\toprule
			$\kappa$ & $N$ & $\mathrm{SU_I(IT)}$ & $\mathrm{SU_I(CPU)}$
			& $\mathrm{SU_P(IT)}$ & $\mathrm{SU_P(CPU)}$\\
			\midrule
			\multirow{4}{*}{$10^4$}
			& $2^{17}$ & 1.72 & 1.95 & 1.76 & 2.83\\
			& $2^{18}$ & 1.54 & 1.49 & 1.39 & 2.16\\
			& $2^{19}$ & 1.48 & 1.69 & 1.45 & 2.66\\
			& $2^{20}$ & 1.38 & 1.62 & 1.25 & 2.38\\
			\midrule
			\multirow{2}{*}{$10^8$}
			& $2^{21}$ & 1.42 & 1.52 & 1.34 & 2.10\\
			& $2^{22}$ & 1.41 & 1.41 & 1.31 & 2.39\\
			\bottomrule
		\end{tabular}
	\end{minipage}
\end{center}

\FloatBarrier
\subsection{Optimization strategies}
We finally examine two practical refinements on the challenging problem with $N=2^{22}$, $d=2^6$, and $\kappa=10^8$. The first adjusts the sketch-size schedule. Under the doubling rule, the last two sketch sizes reach $2^{20}$ and $2^{21}$, substantially reducing the cost advantage of these subproblems. The tuned variant instead uses $m_i\in\{2^k\}_{k=10}^{16}$, followed by a single subproblem of size $2^{20}$ as a transition to the full problem, and sets $a_i=4$. Figure~\ref{fig:tuning} shows that this variant retains the convergence behavior of \spcg\ while reducing its runtime. The schedule is selected empirically to illustrate the potential benefit of adapting the sketch sizes. Automatic selection of the sketch-size schedule remains an open problem.

\begin{figure}[!htbp]
	\centering
	\begin{subfigure}{0.48\linewidth}
		\includegraphics[width=\linewidth]{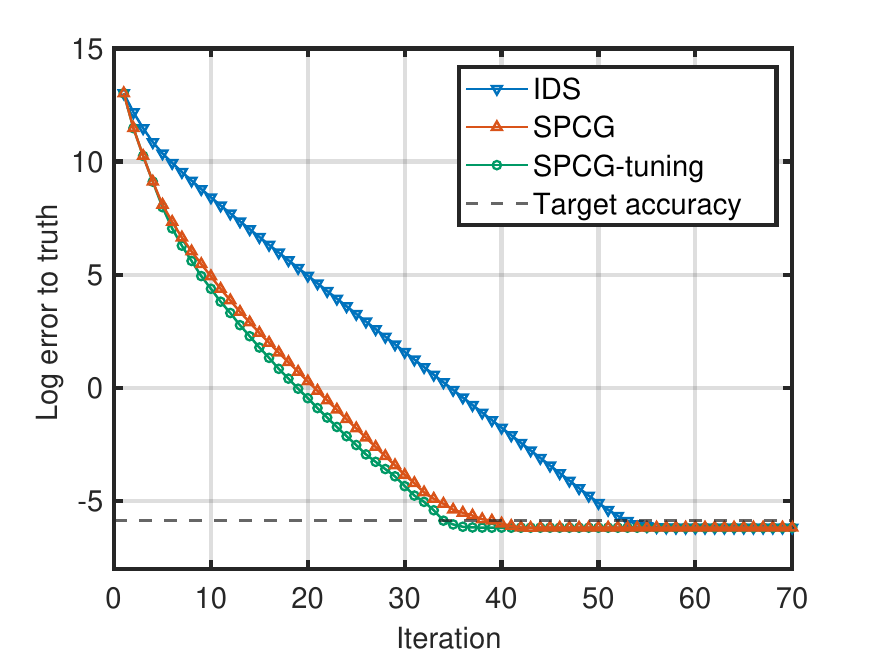}
		\caption{Prediction error.}
	\end{subfigure}
	\hfill
	\begin{subfigure}{0.48\linewidth}
		\includegraphics[width=\linewidth]{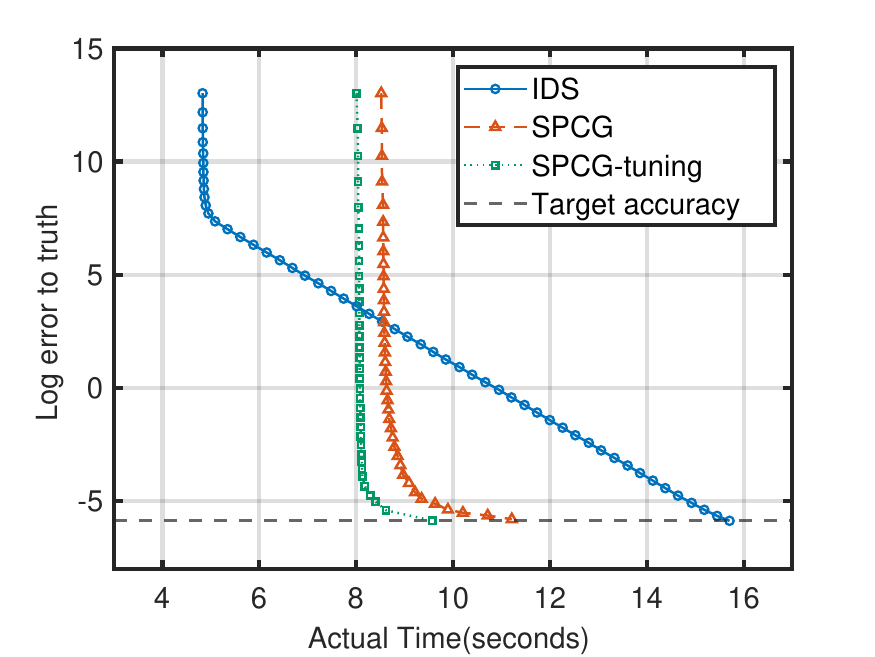}
		\caption{CPU time.}
	\end{subfigure}
	\caption{Comparison of \spcg, its tuned variant, and IDS for $N=2^{22}$, $d=2^6$, and $\kappa=10^8$.}
	\label{fig:tuning}
\end{figure}

The second refinement replaces SRHT sketching with CountSketch in the initialization stage. We denote the two implementations by \spcg-SRHT and \spcg-CS. As discussed following Theorem~\ref{thm:complexity}, CountSketch generally requires a larger sketch size to achieve comparable theoretical embedding guarantees. Nevertheless, Figure~\ref{fig:countsketch} shows that \spcg-CS closely matches the convergence and final accuracy of \spcg-SRHT while substantially reducing the initialization time and total runtime. Thus, CountSketch exhibits practical performance considerably better than its conservative theoretical guarantees suggest.

\begin{figure}[!htbp]
	\centering
	\begin{subfigure}{0.48\linewidth}
		\includegraphics[width=\linewidth]{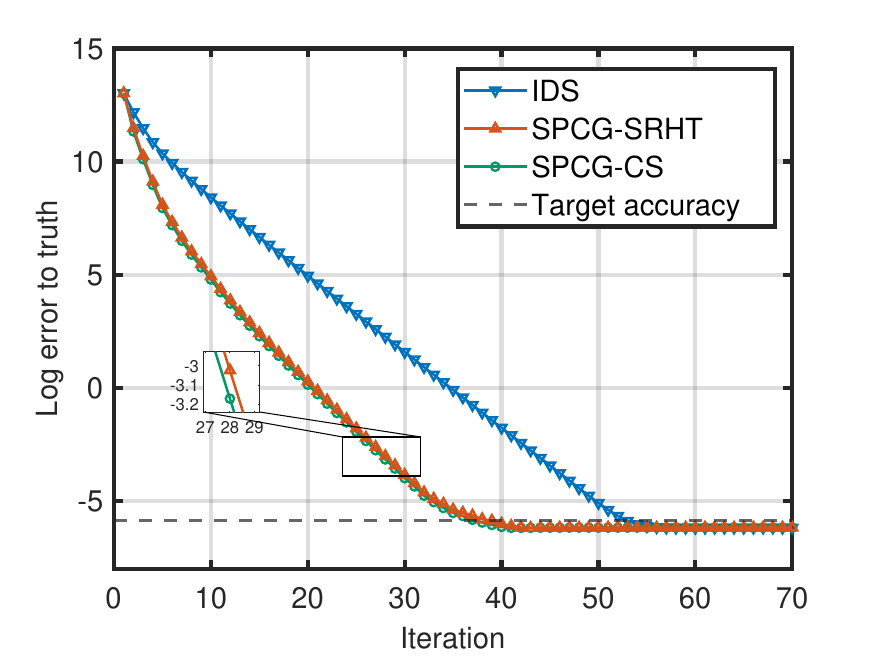}
		\caption{Prediction error.}
	\end{subfigure}
	\hfill
	\begin{subfigure}{0.48\linewidth}
		\includegraphics[width=\linewidth]{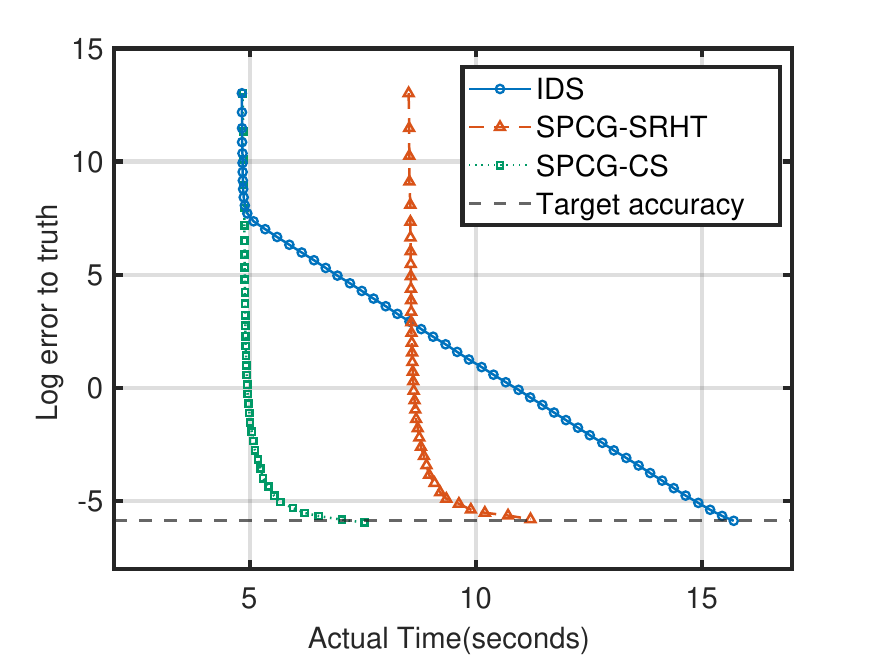}
		\caption{CPU time.}
	\end{subfigure}
	\caption{Comparison of the SRHT and CountSketch variants with IDS for $N=2^{22}$, $d=2^6$, and $\kappa=10^8$.}
	\label{fig:countsketch}
\end{figure}

\section{Conclusions}\label{sec:conclusion}
We proposed SPCG, a two-stage method that sequentially solves increasingly accurate sketched least-squares problems before refining the resulting estimator on the full problem. We established the convergence of \spcg. Numerical experiments demonstrated its computational advantages over IDS and conventional PCG. Exploring more adaptive sketching strategies and efficient iterative solvers based on the framework of \spcg\ remains a valuable direction for future research.

\FloatBarrier
\appendix
\section{Proofs of the main results}\label{app:proofs}

This appendix provides the auxiliary results and complete proofs for Section~\ref{sec:conv}. 

\subsection{Auxiliary lemmas}\label{app:auxiliary}
For an event $\mathcal A$, let $\mathcal A^c$ denote its complement. For $\epsilon\in(0,1)$ and $U\in\reals^{N\times d}$ with orthonormal columns, define
\(
\mathcal E_\epsilon(U):=\bigcap_{i=1}^K\left\{\|U^\top S_i^\top S_iU-I_d\|\leq\epsilon\right\},
\)
\(
\widehat{\mathcal E}_\epsilon(U):=\{\|U^\top\widehat S^\top\widehat SU-I_d\|\leq\epsilon\}.
\)

\begin{lemma}\label{lem:simultaneous-embedding}
	Under the sketch-size conditions of Theorem~\ref{thm:stage}, for any $U\in\reals^{N\times d}$ with orthonormal columns, $\Pr\{\mathcal E_\epsilon(U)^c\}\leq\delta$ and $\Pr\{\widehat{\mathcal E}_\epsilon(U)^c\}\leq\delta$. Consequently, $\Pr\{\mathcal E_\epsilon(U)\cap\widehat{\mathcal E}_\epsilon(U)\}\geq1-2\delta$.
\end{lemma}

\begin{proof}
	For fixed $K$, the constant $c$ in Theorem~\ref{thm:stage} can be chosen sufficiently large so that Lemma~\ref{thm:srht}, applied with failure probability $\delta/K$, gives $\Pr\{\|U^\top S_i^\top S_iU-I_d\|>\epsilon\}\leq\delta/K$ for $i=1,\ldots,K$. Boole's inequality \cite{ross2020first} then yields $\Pr\{\mathcal E_\epsilon(U)^c\}\leq\sum_{i=1}^K\delta/K=\delta$. Applying Lemma~\ref{thm:srht} to $\widehat S$ similarly gives $\Pr\{\widehat{\mathcal E}_\epsilon(U)^c\}\leq\delta$. Thus, another application of Boole's inequality gives
	\(
	\Pr\{\mathcal E_\epsilon(U)\cap\widehat{\mathcal E}_\epsilon(U)\}
	=1-\Pr\{\mathcal E_\epsilon(U)^c\cup\widehat{\mathcal E}_\epsilon(U)^c\}\geq1-2\delta.
	\)
\end{proof}

The following standard CG estimate will be used repeatedly \cite{nocedal2006numerical}.

\begin{lemma}[CG convergence]\label{lem:cg-convergence}
	Let $B\in\reals^{n\times d}$ have full column rank, and apply CG to $B^\top Bx=B^\top b$. If $x_j$ is the $j$th iterate and $x_\star=(B^\top B)^{-1}B^\top b$, then
	\[
	\|B(x_j-x_\star)\|\leq2\left(\frac{\sqrt{\kappa(B^\top B)}-1}{\sqrt{\kappa(B^\top B)}+1}\right)^j\|B(x_0-x_\star)\|.
	\]
\end{lemma}

In the remainder of the appendix, set $U=U_X$ and abbreviate $\mathcal E_\epsilon:=\mathcal E_\epsilon(U_X)$ and $\widehat{\mathcal E}_\epsilon:=\widehat{\mathcal E}_\epsilon(U_X)$.

\begin{lemma}\label{lem:norm-transfer}
	On the event $\mathcal E_\epsilon$, for any $v\in\reals^d$ and $i=1,\ldots,K$, $\sqrt{1-\epsilon}\|Xv\|\leq\|S_iXv\|\leq\sqrt{1+\epsilon}\|Xv\|$.
\end{lemma}

\begin{proof}
	On \(\mathcal E_\epsilon\), we have
	\(
	(1-\epsilon)I_d
	\preceq U_X^\top S_i^\top S_iU_X
	\preceq(1+\epsilon)I_d.
	\)
	Using the compact SVD \(X=U_XD_XV_X^\top\), we obtain
	\begin{equation}\label{eq:app-stage-loewner}
		(1-\epsilon)X^\top X
		\preceq X^\top S_i^\top S_iX
		\preceq(1+\epsilon)X^\top X.
	\end{equation}
	Taking the quadratic forms of \eqref{eq:app-stage-loewner} with an arbitrary
	\(v\in\mathbb R^d\) and then taking square roots proves the result.
\end{proof}

\begin{lemma}\label{lem:preconditioned-hessians}
	On the event
	$\mathcal E_\epsilon\cap\widehat{\mathcal E}_\epsilon$,
	\begin{equation}\label{eq:app-sketched-hessian}
		\frac{1-\epsilon}{1+\epsilon}I_d
		\preceq
		\widehat H^{-1/2}(S_iX)^\top S_iX\widehat H^{-1/2}
		\preceq
		\frac{1+\epsilon}{1-\epsilon}I_d.
	\end{equation}
	On the event $\widehat{\mathcal E}_\epsilon$,
	\begin{equation}\label{eq:app-full-hessian}
		\frac{1}{1+\epsilon}I_d
		\preceq
		\widehat H^{-1/2}X^\top X\widehat H^{-1/2}
		\preceq
		\frac{1}{1-\epsilon}I_d.
	\end{equation}
\end{lemma}

\begin{proof}
	On $\widehat{\mathcal E}_\epsilon$, the SVD argument used above gives
	\begin{equation}\label{eq:app-preconditioner-loewner}
		(1-\epsilon)X^\top X
		\preceq
		X^\top\widehat S^\top\widehat SX
		=\widehat H
		\preceq
		(1+\epsilon)X^\top X.
	\end{equation}
	On $\mathcal E_\epsilon\cap\widehat{\mathcal E}_\epsilon$,
	combining \eqref{eq:app-stage-loewner} and
	\eqref{eq:app-preconditioner-loewner} yields
	\(
	(S_iX)^\top S_iX
	\succeq
	(1-\epsilon)X^\top X
	\succeq
	[(1-\epsilon)/(1+\epsilon)]\widehat H
	\), and
	\(
	(S_iX)^\top S_iX
	\preceq
	(1+\epsilon)X^\top X
	\preceq
	[(1+\epsilon)/(1-\epsilon)]\widehat H.
	\)
	Therefore,
	\(
	[(1-\epsilon)/(1+\epsilon)]\widehat H
	\preceq
	(S_iX)^\top S_iX
	\preceq
	[(1+\epsilon)/(1-\epsilon)]\widehat H.
	\)
	Since $X$ has full column rank and $\epsilon<1$,
	\eqref{eq:app-preconditioner-loewner} implies
	$\widehat H\succ0$. Applying a congruence transformation with
	$\widehat H^{-1/2}$ proves \eqref{eq:app-sketched-hessian}.
	The bound \eqref{eq:app-full-hessian} follows analogously by
	rearranging \eqref{eq:app-preconditioner-loewner} and applying
	the same congruence transformation.
\end{proof}

\subsection{Proof of Theorem~\ref{thm:stage}}\label{app:stage-proof}
For any $v\in\reals^d$, define $\check v=\widehat H^{1/2}v$. The preconditioned normal equation for the $i$th sketched subproblem is
\begin{equation}\label{eq:app-preconditioned-sketched}
	\widehat H^{-1/2}(S_iX)^\top S_iX\widehat H^{-1/2}\check\beta
	=\widehat H^{-1/2}(S_iX)^\top S_iY.
\end{equation}

\begin{proof}
	Set $\check X_i:=S_iX\widehat H^{-1/2}$, which has full column rank on $\mathcal E_\epsilon\cap\widehat{\mathcal E}_\epsilon$. Applying Lemma~\ref{lem:cg-convergence} to \eqref{eq:app-preconditioned-sketched} gives
	\begin{equation}\label{eq:app-cg-sketched}
		\|\check X_i(\check\beta_{a_i}^i-\check{\tilde\beta}^{i})\|
		\leq2q_i^{a_i}\|\check X_i(\check\beta_0^i-\check{\tilde\beta}^{i})\|,
		\qquad
		q_i:=\frac{\sqrt{\kappa_i}-1}{\sqrt{\kappa_i}+1},
	\end{equation}
	where $\kappa_i:=\kappa(\check X_i^\top\check X_i)$. Since $\check X_i\check v=S_iXv$, \eqref{eq:app-cg-sketched} is equivalent to
	$\|S_iX(\beta_{a_i}^i-\tilde\beta^i)\|\leq2q_i^{a_i}\|S_iX(\beta_0^i-\tilde\beta^i)\|$. Lemma~\ref{lem:preconditioned-hessians} gives $\kappa_i\leq[(1+\epsilon)/(1-\epsilon)]^2$. Since \(f(x)=(x-1)/(x+1)\) is increasing on \((-1,+\infty)\), $q_i\leq\frac{(1+\epsilon)/(1-\epsilon)-1}{(1+\epsilon)/(1-\epsilon)+1}=\epsilon$. Combining these bounds with Lemma~\ref{lem:norm-transfer} yields
	\begin{equation}\label{eq:app-stage-intermediate}
		\|X(\beta_{a_i}^i-\tilde\beta^i)\|
		\leq2\sqrt{\frac{1+\epsilon}{1-\epsilon}}\,\epsilon^{a_i}\|X(\beta_0^i-\tilde\beta^i)\|.
	\end{equation}
	Since $\epsilon<1/9$, $2\epsilon\sqrt{(1+\epsilon)/(1-\epsilon)}<\sqrt{5}/9<1/4$ and $\epsilon<1/4$.
	Hence, for any $a_i\geq1$, $2\sqrt{(1+\epsilon)/(1-\epsilon)}\,\epsilon^{a_i}=\{2\epsilon\sqrt{(1+\epsilon)/(1-\epsilon)}\}\epsilon^{a_i-1}\leq(1/4)^{a_i}$. 
	With $\theta=1/4$, \eqref{eq:app-stage-intermediate} therefore reduces to
	\begin{equation}\label{eq:app-stage-exact}
		\|X(\beta_{a_i}^i-\tilde\beta^i)\|\leq\theta^{a_i}\|X(\beta_0^i-\tilde\beta^i)\|.
	\end{equation}
	Applying the triangle inequality to both terms in \eqref{eq:app-stage-exact} gives
	\begin{equation}\label{eq:app-stage-exact2}
		\|X(\beta_{a_i}^i-\beta)\|\leq\theta^{a_i}\|X(\beta_0^i-\beta)\|+(1+\theta^{a_i})\|X(\tilde\beta^i-\beta)\|.
	\end{equation}
	Taking expectations on both sides of \eqref{eq:app-stage-exact2} with respect to the noise and applying Lemma~\ref{lem:simultaneous-embedding} proves that \eqref{eq:stage} holds with probability at least $1-2\delta$.
\end{proof}

\subsection{Proof of Theorem~\ref{thm:global}}\label{app:global-proof}
We first establish the error bounds for the two stages separately.

\begin{lemma}[First-stage error]\label{lem:first-stage}
	Under the assumptions of Theorem~\ref{thm:stage}, with probability at least $1-2\delta$, it holds that
	\(
	\E\|X(\beta_{T^\dagger}-\beta)\|\leq\theta^{T^\dagger}\E\|X(\beta_0-\beta)\|+\sum_{i=1}^K(1+\theta^{a_i})\theta^{T^\dagger-\mathcal T_i}\delta_i.
	\)
\end{lemma}

\begin{proof}
	Let $e_0:=\E\|X(\beta_0-\beta)\|$ and $e_i:=\E\|X(\beta_{a_i}^i-\beta)\|$. Theorem~\ref{thm:stage} gives $e_1\leq\theta^{a_1}e_0+(1+\theta^{a_1})\delta_1$. For $i\geq2$, the warm start gives $e_i\leq\theta^{a_i}e_{i-1}+(1+\theta^{a_i})\delta_i$. At the second subproblem, this yields $e_2\leq\theta^{a_1+a_2}e_0+(1+\theta^{a_1})\theta^{a_2}\delta_1+(1+\theta^{a_2})\delta_2$.
	Proceeding inductively, for any $i=1,\ldots,K$, we obtain
	\(
	e_i\leq\theta^{\mathcal T_i}e_0+\sum_{\ell=1}^i(1+\theta^{a_\ell})\theta^{\mathcal T_i-\mathcal T_\ell}\delta_\ell.
	\)
	Setting $i=K$ and using $\mathcal T_K=T^\dagger$ and $\beta_{T^\dagger}=\beta_{a_K}^K$ leads to the result.
\end{proof}

\begin{lemma}[Second-stage error]\label{lem:full-pcg}
	Under the assumptions of Theorem~\ref{thm:stage}, with probability at least $1-2\delta$, for any $T\geq T^\dagger$,
	\(
	\E\|X(\beta_T-\hat\beta)\|\leq\theta^{T-T^\dagger}\E\|X(\beta_{T^\dagger}-\hat\beta)\|.
	\)
\end{lemma}

\begin{proof}
	In the second stage, PCG is applied to 
	$\widehat H^{-1/2}X^\top X\widehat H^{-1/2}\check\beta
	=\widehat H^{-1/2}X^\top Y$.
	As in the proof of Theorem~\ref{thm:stage}, let
	$A:=\widehat H^{-1/2}X^\top X\widehat H^{-1/2}$ and
	$q:=(\sqrt{\kappa(A)}-1)/(\sqrt{\kappa(A)}+1)$.
	After $s=T-T^\dagger$ iterations, Lemma~\ref{lem:cg-convergence} implies
	$\|X(\beta_T-\hat\beta)\|
	\leq2q^s\|X(\beta_{T^\dagger}-\hat\beta)\|$.
	By Lemma~\ref{lem:preconditioned-hessians},
	$\kappa(A)\leq[(1+\epsilon)/(1-\epsilon)]^2$, which implies
	$q\leq\epsilon$. For $s\geq1$, the condition $\epsilon<1/9$ ensures
	$2\epsilon^s\leq(1/4)^s=\theta^s$, while the case $s=0$ is immediate.
	Taking expectations with respect to the noise and applying
	Lemma~\ref{lem:simultaneous-embedding} establishes the result with
	probability at least $1-2\delta$.
\end{proof}

\begin{proof}[Proof of Theorem~\ref{thm:global}]
	By Lemma~\ref{lem:full-pcg} and the triangle inequality,
	\begin{align*}
		\E\|X(\beta_T-\beta)\|
		&\leq \E\|X(\beta_T-\hat\beta)\|+\delta_{\rm OLS} 
		\leq \theta^s\E\|X(\beta_{T^\dagger}-\hat\beta)\|
		+\delta_{\rm OLS} 
		\leq \theta^s\E\|X(\beta_{T^\dagger}-\beta)\|
		+(1+\theta^s)\delta_{\rm OLS}.
	\end{align*}
	Combining this estimate with Lemma~\ref{lem:first-stage} and using
	$s+T^\dagger=T$ results in
	\begin{align*}
		\E\|X(\beta_T-\beta)\|
		\leq\theta^T\E\|X(\beta_0-\beta)\|
		+\sum_{i=1}^K(1+\theta^{a_i})\theta^{T-\mathcal T_i}\delta_i
		+(1+\theta^s)\delta_{\rm OLS},
	\end{align*}
	which proves \eqref{eq:global-exact}. Since $\delta_i\leq C\delta_{\rm OLS}$ and $a_i\geq1$, $T^\dagger-\mathcal T_i=\sum_{\ell=i+1}^K a_\ell\geq K-i$ and
	$1+\theta^{a_i}\leq1+\theta$. Consequently,
	\begin{align*}
		\sum_{i=1}^K(1+\theta^{a_i})
		\theta^{T^\dagger-\mathcal T_i}\delta_i
		&\leq C\delta_{\rm OLS}(1+\theta)
		\sum_{j=0}^{K-1}\theta^j 
		<\frac{1+\theta}{1-\theta}C\delta_{\rm OLS}
		=\frac{5}{3}C\delta_{\rm OLS}.
	\end{align*}
	Since $T-\mathcal T_i=s+T^\dagger-\mathcal T_i$, substituting this
	estimate into \eqref{eq:global-exact} establishes
	\eqref{eq:global-ols}.
\end{proof}

\subsection{Proof of Theorem~\ref{thm:iteration-bounds}}\label{app:iteration-proof}
We first recall the asymptotic prediction-efficiency result that motivates the comparison between consecutive sketched estimators.

\begin{lemma}[\cite{dobriban2019asymptotics}]\label{prop:prediction-efficiency}
	Suppose that the empirical spectral distribution of $X^\top X$ converges weakly to a compactly supported distribution bounded away from zero. If $d/N\to\gamma\in(0,1)$ and $m_i/N\to\xi_i\in(\gamma,1)$ as $N\to\infty$, then the prediction efficiency of the $i$th SRHT estimator satisfies
	\[
	\mathrm{PE}(i):=\frac{\E\|X(\tilde\beta^i-\beta)\|^2}{\E\|X(\hat\beta-\beta)\|^2}\longrightarrow\frac{1-\gamma}{\xi_i-\gamma}.
	\]
\end{lemma}

\begin{proof}
	The result follows by specializing Theorem~2.4 of \cite{dobriban2019asymptotics} to SRHT sketch-and-solve estimators.
\end{proof}
Let $g(\gamma,\xi)=(1-\gamma)/(\xi-\gamma)$, which is increasing in $\gamma$ and decreasing in $\xi$. Hence, when $\xi_i>m_i/N>d/N>\gamma$, Lemma~\ref{prop:prediction-efficiency} leads to the finite-sample relation $\mathrm{PE}(i)\leq g(d/N,m_i/N)=(N-d)/(m_i-d)$. Equivalently, letting $Z_i=\|X(\tilde\beta^i-\beta)\|$, $Z_{\rm OLS}=\|X(\hat\beta-\beta)\|$, and $c_i=\sqrt{(N-d)/(m_i-d)}$, we have $\E Z_i^2\leq c_i^2\E Z_{\rm OLS}^2$. If $\operatorname{Var}(Z_i)\geq\operatorname{Var}(c_iZ_{\rm OLS})$, then
\(
(\E Z_i)^2=\E Z_i^2-\operatorname{Var}(Z_i)
\leq c_i^2\bigl[\E Z_{\rm OLS}^2-\operatorname{Var}(Z_{\rm OLS})\bigr]
=c_i^2(\E Z_{\rm OLS})^2,
\)
and therefore $\delta_i\leq c_i\delta_{\rm OLS}$. The subsequent discussions are based on the following two conditions:
\begin{itemize}
	\item $\xi_i>m_i/N>d/N>\gamma$;
	\item $\operatorname{Var}(Z_i)\geq\operatorname{Var}(c_iZ_{\rm OLS})$.
\end{itemize}

\begin{proof}[Proof of Theorem~\ref{thm:iteration-bounds}]
	Fix $i\in\{2,\ldots,K\}$. Apply the preceding comparison to the consecutive estimators $\tilde\beta^{i-1}$ and $\tilde\beta^i$, with $\tilde\beta^i$ serving as the reference estimator and $(m_i,m_{i-1})$ replacing $(N,m_i)$. The monotonicity of $g$ and the first condition imply $\E Z_{i-1}^2\leq r_{i-1,i}^2\E Z_i^2$, while the second condition implies $\operatorname{Var}(Z_{i-1})\geq r_{i-1,i}^2\operatorname{Var}(Z_i)$. Consequently,
	\begin{align*}
		\delta_{i-1}^2
		&=\E Z_{i-1}^2-\operatorname{Var}(Z_{i-1}) 
		\leq r_{i-1,i}^2\bigl[\E Z_i^2-\operatorname{Var}(Z_i)\bigr]
		=r_{i-1,i}^2\delta_i^2.
	\end{align*}
	Since $\delta_{i-1}$ and $\delta_i$ are nonnegative, $\delta_{i-1}\leq r_{i-1,i}\delta_i$. By the warm-start relation $\beta_0^i=\beta_{a_{i-1}}^{i-1}$ and the triangle inequality,
	\begin{equation}\label{eq:app-warm-bound}
		\E\|X(\beta_0^i-\tilde\beta^i)\|\leq\E\|X(\beta_{a_{i-1}}^{i-1}-\tilde\beta^{i-1})\|+\delta_{i-1}+\delta_i.
	\end{equation}
	The stagewise accuracy of subproblem $i-1$ and the interstage relation established above imply
	\begin{equation}\label{eq:app-warm-final}
		\E\|X(\beta_0^i-\tilde\beta^i)\|
		\leq(1+\omega)\delta_{i-1}+\delta_i
		\leq[(1+\omega)r_{i-1,i}+1]\delta_i.
	\end{equation}
	Taking expectations in \eqref{eq:app-stage-exact} and applying \eqref{eq:app-warm-final}, we obtain
	\[
	\E\|X(\beta_{a_i}^i-\tilde\beta^i)\|
	\leq\theta^{a_i}[(1+\omega)r_{i-1,i}+1]\delta_i.
	\]
	Thus, the stagewise accuracy condition is satisfied whenever $\theta^{a_i}[(1+\omega)r_{i-1,i}+1]\leq\omega$. Since $\theta=1/4$, taking logarithms establishes the stated lower bound. Its positivity follows from $\omega\in(0,1)$ and $(1+\omega)r_{i-1,i}+1>1$.
\end{proof}

\subsection{Proof of Theorem~\ref{thm:complexity}}\label{app:complexity-proof}

\begin{proof}[Proof of Theorem~\ref{thm:complexity}]
	For $i=2,\ldots,K$, the doubling schedule implies
	\(
	r_{i-1,i}^2=(m_i-d)/(m_{i-1}-d)
	=2+d/(m_{i-1}-d).
	\)
	Thus, $r_{i-1,i}$ decreases with $i$ and satisfies $0<r_{i-1,i}\leq r_{1,2}=\bar r$. Since the lower bound in Theorem~\ref{thm:iteration-bounds} is increasing in $r_{i-1,i}$, setting $a_i$ equal to this unrounded bound gives
	\[
	\frac{\ln(1/\omega)}{\ln4}<a_i\leq
	\frac{\ln\{[1+(1+\omega)\bar r]/\omega\}}{\ln4}
	=\alpha,
	\qquad i=2,\ldots,K.
	\]
	For the first subproblem, \eqref{eq:app-stage-exact} and the definition of $q_1$ imply
	\(
	\E\|X(\beta_{a_1}^1-\tilde\beta^1)\|
	\leq\theta^{a_1}\E\|X(\beta_0-\tilde\beta^1)\|
	=\theta^{a_1}q_1\delta_1.
	\)
	The stagewise accuracy condition therefore holds if $\theta^{a_1}q_1\leq\omega$. With $\theta=1/4$, the smallest unrounded sufficient value is $a_1=\ln(q_1/\omega)/\ln4$. The assumption $1<q_1<1+(1+\omega)\bar r$ then implies $\ln(1/\omega)/\ln4<a_1<\alpha$. Consequently, $\ln(1/\omega)/\ln4<a_i\leq\alpha$ for any $i=1,\ldots,K$.
	
	We next consider the dominant cost of each stage. The nested sketches $(S_iX,S_iY)$ require a single Hadamard transform with cost $O(Nd\log_2N)$. Constructing and factorizing $\widehat H$ requires $O(rd^2+d^3)$ additional operations. Hence, the initialization cost is $O(Nd\log_2N+rd^2+d^3)$, with dominant term $Nd\log_2N$ in the regime considered here.
	
	One PCG iteration on the $i$th sketched subproblem requires $\{(4m_i+13)d+2d^2\}$ FLOPs. Thus, the first-stage cost is $\sum_{i=1}^Ka_i\{(4m_i+13)d+2d^2\}$. Since $\sum_{i=1}^Km_i=(1-2^{-K})N<N$, $a_i\leq\alpha$, and $K=O(1)$, its dominant term is bounded by $4\alpha Nd$.
	
	Finally, since $\theta=1/4$, the stopping condition $\theta^T\leq\sigma$ corresponds to the smallest unrounded total iteration count $T=\ln(1/\sigma)/\ln4$. Using $a_i>\ln(1/\omega)/\ln4$ for any $i=1,\ldots,K$, we obtain
	\[
	T-T^\dagger
	=\frac{\ln(1/\sigma)}{\ln4}-\sum_{i=1}^Ka_i
	<\frac{\ln(1/\sigma)-K\ln(1/\omega)}{\ln4}=\frac{\ln(\omega^K/\sigma)}{\ln4}.
	\]
	Each full-data PCG iteration has dominant cost $4Nd$ FLOPs. Therefore, the second-stage cost is bounded by $4[\ln(\omega^K/\sigma)/\ln4]Nd$, which is positive when $\omega>\sigma^{1/K}$.
\end{proof}

\section*{Acknowledgments}
The corresponding author is supported by the National Natural Science Foundation of China (Grant Nos. 12071215 and 11101213).

\section*{Data availability}
No data was used for the research described in the article.

\printcredits

\nocite{*}
\bibliographystyle{cas-model2-names}
\bibliography{SPCG-REF}



\end{document}